\documentclass[12pt,english]{amsart}

\usepackage{amsmath,amscd,amsfonts,amsthm,amssymb}
\usepackage{graphicx}
\usepackage{verbatim} 
\input xypic
\xyoption{all}

\newcommand{\shrinkmargins}[1]{
  \addtolength{\textheight}{#1\topmargin}
  \addtolength{\textheight}{#1\topmargin}
  \addtolength{\textwidth}{#1\oddsidemargin}
  \addtolength{\textwidth}{#1\evensidemargin}
  \addtolength{\topmargin}{-#1\topmargin}
  \addtolength{\oddsidemargin}{-#1\oddsidemargin}
  \addtolength{\evensidemargin}{-#1\evensidemargin}
  }

\shrinkmargins{.7}
\newcommand{\field}[1]{\mathbb{#1}}

\newcommand{\Z}{\field{Z}}
\newcommand{\F}{\field{F}}

\newcommand{\R}{\field{R}}

\newcommand{\beq}{\begin{displaymath}}
\newcommand{\eeq}{\end{displaymath}}
\newcommand{\beqn}{\begin{equation}}
\newcommand{\eeqn}{\end{equation}}

\theoremstyle{plain}
\newtheorem{thm}{Theorem}[section]
\newtheorem{prop}[thm]{Proposition}

\newtheorem{lem}[thm]{Lemma}
\newtheorem*{nota}{Notation}

\newtheorem*{cor*}{Corollary}
\newtheorem*{defn*}{Definition}

\theoremstyle{definition}

\theoremstyle{remark}
\newtheorem{rem}[thm]{Remark}
\newtheorem*{rem*}{Remark}

\newenvironment{quest}[2][Question]{\begin{trivlist}
\item[\hskip \labelsep {\bfseries #1}\hskip \labelsep {\bfseries #2.}]}{\end{trivlist}}

\newenvironment{thmcustom}[2][Theorem]{\begin{trivlist}
\item[\hskip \labelsep {\bfseries #1}\hskip \labelsep {\bfseries #2.}]}{\end{trivlist}}

\begin{document}
\title{Kakeya sets over non-archimedean local rings}

\author{Evan P. Dummit, M\'arton Hablicsek}
\address{UW-Madison, Dept. of Mathematics, 480 Lincoln Dr., Madison, WI 53706-1388}
\email{dummit@math.wisc.edu}
\email{hablics@math.wisc.edu}

\maketitle

\begin{abstract}
In a recent paper \cite{Ell09} of Ellenberg, Oberlin, and Tao, the authors asked whether there are Besicovitch phenomena in $\F_q[[t]]^n$. In this paper, we answer their question in the affirmative by explicitly constructing a Kakeya set in $\F_q[[t]]^n$ of measure 0.  Furthermore, we prove that any Kakeya set in $\F_q[[t]]^2$ or $\Z_p^2$ is of Minkowski dimension 2.
\end{abstract}

\section{Introduction}
The classical Kakeya problem asks how small (in measure) a subset of the Euclidean plane can be, if it contains a unit line segment in every possible direction; Besicovitch \cite{Bes19} showed that such sets could have measure 0.  Wolff \cite{Wol99} proposed a generalization of the Besicovitch-Kakeya problem to a finite field setting, which was solved by Dvir in \cite{Dvir08}.  We give a natural generalization of the Besicovitch-Kakeya problem to any infinite ring admitting a finite Haar measure.
\\ \\
Let $R$ be an infinite ring admitting a Haar measure $\mu$ such that $\mu(R)$ is finite.
\begin{defn*}
A {\bf line} with direction vector $v \in R^n$ through the point $x \in R^n$ consists of the elements in $R^n$ of the form $x + t v$ as $t$ runs through the elements of $R$.
\end{defn*}

\begin{defn*}
A {\bf Kakeya set} in $R^n$ is a subset of $R^n$ which contains (all the points on) at least one line with each possible direction vector.
\end{defn*}
\vspace{1pc}
The generalization of the classical Besicovitch-Kakeya problem is then the following:
\begin{quest}{1}
If $R$ is an infinite ring admitting a finite Haar measure and $n>1$, does there exist a Kakeya set of zero Haar measure in $R^n$?
\end{quest}
\vspace{1pc}
If such a Kakeya set of measure zero exists, we refer to it as a Besicovitch set.  For such sets, the analogue of the so-called ``Kakeya conjecture'' becomes
\begin{quest}{2}
If there exists a Besicovitch set in $R^n$, is its dimension necessarily equal to $n$?
\end{quest}
\vspace{1pc}
In \cite{Ell09}, Ellenberg, Oberlin, and Tao posed Question 1 using the outer measure for $R = \F_q[[t]]$:
\\
\textit{
Let $E$ be a subset of $\F_q[[t]]^n$ containing a line in every direction, and write $E_k$ for the image of $E$ under the projection $\F_q[[t]]^n\rightarrow (\F_q[[t]]/t^k)^n$. Are there Besicovitch phenomena in $\F_q[[t]]^n$? That is, is it possible that
$$\lim_{k\rightarrow  \infty} |E_k||\F_q|^{-nk} = 0?$$
}

In Section 2 we show that the answer to Question 1 is ``yes'' for $R = \F_q[[t]]$ by proving the following:

\begin{thm}\label{main}
For all $n>1$, there exists a subset $E\subset \F_q[[t]]^n$ of measure 0 which contains a line in every direction.
\end{thm}

In Section 3 we show that the answer to Question 2 is ``yes'' for both $R = \F_q[[t]]$ and $R = \Z_p$, in dimension 2 and with the Minkowski dimension:

\begin{thm}\label{Mink}
Let $E$ be a Kakeya set in $R^2$ where $R=\Z_p$ or $R=\F_{q} [[t]]$.  Then $E$ has Minkowski dimension 2.
\end{thm}

\begin{rem*} For a finite ring $R$, the Minkowski dimension of a set $E\subset R^n$ is defined as $\displaystyle\frac{\log |E|}{\log |R|}$. The natural analogue of the Minkowski dimension of a compact subset $E\subset \F_q[[t]]^n$ is
$$\lim_{k\rightarrow \infty} \frac{\log |E_k|}{\log |\F_q|^k},$$
provided the limit exists.
\end{rem*}

We note that the above statements are analogous to the known results for the original Besicovitch-Kakeya problem over $\R$; cf. \cite{Bes19}, \cite{Dav71}.

\section{Existence of a Kakeya set of measure 0 for $R=\F_q[[t]]$}

Let $\F_q$ be the finite field of order $q$.  We refer to a nonzero direction vector $v=(a,b)$ in $\F_q[[t]]^2$ as {\bf nonreduced} if $t$ divides both $a$ and $b$, and as {\bf reduced} otherwise.  It is obvious that any line with nonreduced direction vector $v$ passing through $(x,y)$ is contained in the line with direction vector $v/t$ through $(x,y)$; thus, we need only consider reduced direction vectors.

\begin{lem}\label{general_Kakeya_from_specific} If $H$ is a set of measure 0 in $\F_q[[t]]^2$ containing some line with direction vector $(1,b)$ for each $b\in \F_q [[t]]$, then $K:=\{(x,y): (x,y)\text{ or }(y,x)\in H\}$ is a Besicovitch set in $\F_q[[t]]^2$, and $K^{(n)}:=K\times \F_q[[t]]^{n-2}$ is a Besicovitch set in $\F_q[[t]]^n$.
\end{lem}
\begin{proof}
Each reduced direction vector in $\F_q[[t]]^2$ is of the form $(a,1)$ or $(1,a)$ for some $a\in \F_q [[t]]$.  Therefore, if $H$ is of measure 0 and contains a line with direction vector $(1,b)$ for each $b\in \F_q [[t]]$, then $K$ is also of measure 0 and contains a line in every direction; i.e., $K$ is a Besicovitch set in $\F_q[[t]]^2$.  Furthermore, $K^{(n)}$ is also of measure 0 and will contain a line in every direction; hence is a Besicovitch set in $\F_q[[t]]^n$ 
\end{proof}

We employ the following notation:
\begin{itemize}
\item For any $a\in \F_q [[t]]$, $a_i$ denotes the coefficient of $t^i$.
\item For any $a\in \F_q [[t]]$, the element $a^* \in \F_q [[t]]$ is defined by 
$$a^*_i = \begin{cases} 0 & \mbox{if } i=2^k-2 \mbox{ for some natural number } k,\\ a_{i+1} & \mbox{otherwise}.\end{cases}$$
\end{itemize}

\begin{thm}\label{thm_construction_of_specific_Kakeya} The set $$H :=\{(x,y)\in \F_q[[t]]^2:\, ax+y=a^* \mbox{ for some } a\in \F_q[[t]]\}$$
contains a line with direction vector $(1,b)$ for each $b\in \F_q [[t]]$, and has measure 0.
\end{thm}
By definition, for any $b\in \F_q[[t]]$, the points $(x,y) = (0,-b^*) + s(1,b)$ are contained in $H$, as $s$ ranges over $\F_q[[t]]$, since $(-b)s+(-b^* +bs) = (-b)^*$.  Therefore, we see that for any $b \in \F_q [[t]]$, a line with direction vector $(1,b)$ is contained in $H$, and so all that remains to be proven is that $\mu(H)=0$.
\\

We can see that $(x,y)\in H$ if and only if there exist $a_i \in \F_q$ for all $i\in \mathbb{N}$ such that the coefficients of $x$ and $y$ satisfy the following infinite system:
\begin{align*}
& a_0x_0+y_0=0, \tag*{[0]}\\
& a_1x_0+a_0x_1+y_1=a_2, \tag*{[1]}\\
& a_2x_0+a_1x_1+a_0x_2+y_2=0, \tag*{[2]}\\
& \qquad \qquad \qquad \vdots \\
& a_nx_0+a_{n-1}x_1+\dots+a_0x_n+y_n=a^*_n, \tag*{[$n$]} \\
& \qquad \qquad \qquad \vdots
\end{align*}
\\

We refer to $a_lx_0+\dots+a_0x_l+y_l=a_l^*$ as Equation $[l]$ and the system of Equations $[0]$ through $[l]$ as the ``$l$th-stage system''.
\\

For an arbitrary element $(x,y)\in \F_q[[t]]^2$, define $s_n(x,y)$ to be the number of tuples $(a_0,\dots,a_n)\in \F_q^{n+1}$ satisfying the $n$th-stage system; observe that $s_n(x,y)$ only depends on $\{x_0,y_0,\dots,x_n,y_n\}$.
Clearly if $s_n(x,y)=0$ for any integer $n$, then $(x,y)\not\in H$, and $s_k(x,y)=0$ for all $k>n$, so $\mu\left(\{(x,y)|s_n(x,y)=0\}\right)$ is non-decreasing as $n\rightarrow \infty$.  We will prove that $\mu\left(\{(x,y)\in \F_q[[t]]^2 \,:\, s_n(x,y)=0\}\right)\rightarrow 1$ as $n\rightarrow \infty$, from which $\mu(H)=0$.
\\

Observe that the equations at any stage are linear in the $a_i$.  Moreover, for $i$ not of the form $2^k-2$ for some integer $k$, Equation $[i]$ states $a_{i+1}=a_ix_0+a_{i-1}x_1+...+a_0x_i+y_i$, and so we may reduce our system of equations by eliminating $a_{i+1}$.  Basic linear algebra then implies that $s_{2^k-2} (x,y)$ is either zero or $q^l$ for some integer $l\leq k$.

\begin{nota}
In the sequel, $C_n$ will denote (any) polynomial of the variables $x_0$, $x_1$, ..., $x_n$, $y_0$, $y_1$, ..., $y_n$; thus for example $C_n + y_{n+1} C_{n-1} = C_{n+1}$ and $C_n + C_n = C_n$. We additionally will use $R_n$ to denote any rational functions of the same variables whose denominator is nonzero everywhere, and $C_{-1}$ to denote (any) element of $\F_q$.
\end{nota}

\begin{lem}\label{decomposition_1}
Every $a_n$ can be written as
$$a_n=\sum_{\substack{k\\ 2^k-1\leq n}} C_{n-2^k} a_{2^k-1}+C_{n-1}.$$
\end{lem}

\begin{proof}
We show the result by strong induction on $n$.  For $n=2^k-1$ the statement is immediate since $a_n=a_{2^k-1}$. Now if $n\neq 2^k-1$, Equation $[n-1]$ reads
$$a_{n-1}x_0+...+a_0x_{n-1}+y_{n-1}=a_n.$$
By the induction hypothesis, every $a_l$ on the left-hand side is of the form
$$a_l=\sum_{\substack{k\\ 2^k-1\leq l}} C_{l-2^k} a_{2^k-1}+C_{l-1}$$
and so we may write
\begin{align*}
a_n &=\sum_{l=0}^{n-1} x_{n-l-1}a_l + y_{n-1} \\
 & =\sum_{l=0}^{n-1} x_{n-l-1}\left(\sum_{\substack{k\\ 2^k-1\leq l}} C_{l-2^k} a_{2^k-1}+C_{l-1}\right)+y_{n-1} \\
&=\sum_{\substack{k\\ 2^k-1\leq n}} a_{2^k-1}\left(\sum_{l=2^k-1}^{n-1} x_{n-l-1}C_{l-2^k}\right)+\sum_{l=0}^{n-1} x_{n-l-1}C_{l-1}+y_{n-1}.
\end{align*}
By definition we have $\displaystyle\sum_{l=2^k-1}^{n-1} x_{n-l-1}C_{l-2^k}=C_{n-2^k}$ and $\displaystyle\sum_{l=0}^{n-1} x_{n-l-1}C_{l-1}+y_{n-1}=C_{n-1}$.  Substituting these expressions yields the result as claimed.
\\
\end{proof}

\begin{lem}\label{decomposition_2}
For $n>0$, the left-hand side of Equation $[2^n-2]$ can be written as
$$\left[ \sum_{0\leq k < n} a_{2^k-1}\left( x_{2^n-2-(2^k-1)}+C_{2^n-2-(2^k-1)-1}\right) \right]+(y_{2^n-2}+C_{2^n-3}).$$
\end{lem}

\begin{proof}
The left-hand side of Equation $[2^n-2]$ is
\begin{equation}\label{lhs_1}
\sum_{0\leq l\leq 2^n-2}a_{l}x_{2^n-2-l}+y_{2^n-2}=\sum_{0\leq k<n} a_{2^k-1}x_{2^n-2-(2^k-1)} + \sum_{\substack{0\leq l\leq 2^n-2\\ l\ne 2^k-1}}a_lx_{2^n-2-l}+y_{2^n-2}.
\end{equation}
By using Lemma \ref{decomposition_1}, \eqref{lhs_1} becomes
\begin{equation}\label{lhs_2}
\sum_{0\leq k<n} a_{2^k-1}x_{2^n-2-(2^k-1)} +\sum_{\substack{0\leq l\leq 2^n-2\\ l\ne 2^k-1}} x_{2^n-2-l} \left(\sum_{2^k-1\leq l} C_{l-2^k}a_{2^k-1}+C_{l-1}\right)+y_{2^n-2}.
\end{equation}
We may rearrange the middle term in \eqref{lhs_2} to obtain
\begin{equation}\label{lhs_3}
\sum_{0\leq k<n} a_{2^k-1}\left(\sum_{2^k-1< l\leq 2^n-2} C_{l-2^k} x_{2^n-2-l}\right)+\sum_{0\leq l\leq 2^n-2} x_{2^n-2-l} C_{l-1}.
\end{equation}
Clearly we have
$$\sum_{2^k-1< l\leq 2^n-2} C_{l-2^k} x_{2^n-2-l}=C_{2^n-2^k-2}$$
and
$$\sum_{0\leq l\leq 2^n-2} x_{2^n-2-l}\left(\sum_{2^k-1\leq l} C_{l-1}\right)=C_{2^n-3},$$
so, upon substituting into \eqref{lhs_3}, we obtain the desired result.
\end{proof}

Now we use the structure of individual terms in Equation $[2^n-2]$, as determined in Lemma \ref{decomposition_2}, to compute the distribution of values for the $s_{2^k-2}(x,y)$ as $x$ and $y$ range over all of $\F_q[[t]]$.

\begin{lem}\label{probabilities}
If for a given $x$ and $y$ we have $s_{2^n-2}(x,y)=q^l$, and \linebreak ${(x_{2^n-1},y_{2^n-1},...,x_{2^{n+1}-2},y_{2^{n+1}-2})}$ are randomly and uniformly chosen from $\F_q$, then 
\begin{center}
$s_{2^{n+1}-2}(x,y)=\begin{cases}
0 & \mbox{with probability }\frac{q-1}{q^{l+2}},\\
q^l & \mbox{with probability }1-\frac{1}{q^{l+1}},\\
q^{l+1} & \mbox{with probability }\frac{1}{q^{l+2}}.\\
\end{cases}$
\end{center}
\end{lem}

\begin{proof}
Since $s_{2^n-2}(x,y)=q^l$, there are $l$ indices $0<n_1<n_2<...<n_l<n$ such that every $a_j$ ($0\leq j\leq 2^n-2)$ can be written as
\begin{equation}\label{basis_elements}
a_j=\sum_{n_i:2^{n_i}-1\leq j} R_{2^n-2} a_{2^{n_i}-1}+R_{2^n-2},
\end{equation}
where we have noted that $a_{2^j-1}$ can be expressed in terms of the $a_{2^i-1}$ with $i\leq j$.  Note that upon eliminating variables, we may introduce rational functions (which may be different for different choices of the parameters $x_i$ and $y_i$), but by the assumption $s_{2^n-2}(x,y)=q^l$, their denominators are nonzero.
\\
Observe that the coefficients $a_{n_1}, \dots , a_{n_l}$ give a basis for the solution space of the $(2^{n+1}-2)$nd stage system.  We will determine the probability that there is a solution to the $a_{2^{n+1}-1}$-st stage system, and the probability that $a_{2^{n+1}-1}$ is linearly independent from these $a_{n_i}$.

Lemma \ref{decomposition_2} states that Equation $2^{n+1}-2$ reads
\begin{align}
0 = \,\, & a_{2^n-1}(x_{2^{n+1}-2^n-1}+C_{2^{n+1}-2-(2^n-1)-1}) \label{LHS} \\
& \qquad + \sum_{0\leq k< n} a_{2^k-1}(x_{2^{n+1}-2^k-1}+C_{2^{n+1}-2-(2^k-1)-1})+(y_{2^{n+1}-2}+C_{2^{n+1}-3}) \notag
\end{align}
\\
Using \eqref{basis_elements}, for the terms $a_{2^k-1}$ where $k\ne n_i$, in the sum in \eqref{LHS} yields
\begin{align*}
& \sum_{\substack{0\leq k< n\\k\ne n_i}} a_{2^k-1}(x_{2^{n+1}-2^k-1}+R_{2^{n+1}-2-(2^k-1)-1}) \\
& \qquad =\sum_{\substack{0\leq k< n\\k\ne n_i}} \left((x_{2^{n+1}-2^k-1}+R_{2^{n+1}-2-(2^k-1)-1})\left(\sum_{n_i< k}R_{2^n-2} a_{2^{n_i}-1}+R_{2^n-2}\right)\right) \\
& \qquad = \sum_{1\leq i\leq l} a_{2^{n_i}-1}  \left(\sum_{n_i< k<n} R_{2^n-2}\left(x_{2^{n+1}-2^k-1}+R_{2^{n+1}-2-(2^k-1)-1}\right)\right) \\
&\qquad \qquad + \sum_{\substack{0\leq k<n\\k\ne n_i}} R_{2^n-2} \left(x_{2^{n+1}-2^k-1}+R_{2^{n+1}-2-(2^k-1)-1}\right).
\end{align*}
Since $k<n$ we have $2^{n+1}-2^k-1<2^{n+1}-2^n-1=2^n-1$, hence the above sum may be written more compactly as
$$\sum_{1\leq i\leq l} a_{2^{n_i}-1}R_{2^n-2}+R_{2^n-2}.$$

Therefore, again because $k<n$, \eqref{LHS} can be written as:
\begin{align}
0 =\,\,& a_{2^n-1}(x_{2^{n+1}-2^n-1}+R_{2^{n+1}-2-(2^n-1)-1}) \label{new_equation} \\
& \qquad + \sum_{n_i} a_{2^{n_i}-1} (x_{2^{n+1}-2^{n_i}-1}+R_{2^{n+1}-2-(2^{n_i}-1)-1})+(y_{2^{n+1}-2}+R_{2^{n+1}-3}). \notag
\end{align}
Now we analyze the right-hand side 
\begin{align}
\,\,& a_{2^n-1}(x_{2^{n+1}-2^n-1}+R_{2^{n+1}-2-(2^n-1)-1}) \label{newLHS} \\
& \qquad + \sum_{n_i} a_{2^{n_i}-1} (x_{2^{n+1}-2^{n_i}-1}+R_{2^{n+1}-2-(2^{n_i}-1)-1})+(y_{2^{n+1}-2}+R_{2^{n+1}-3}). \notag
\end{align}
of \eqref{new_equation}, to determine for which values of the $a_{n_i}$ it takes the value zero.
\\
There are three cases: \eqref{newLHS} can be `identically zero' (i.e., zero for all values of the $a_{n_i}$), `never zero' (i.e., zero for no possible values of the $a_{n_i}$), or `sometimes zero' (i.e., if the coefficient of $a_{2^n-1}$ or one of the $a_{2^{n_i}-1}$ is not zero).
\begin{enumerate}
\item \textbf{\eqref{newLHS} is identically zero.} Since $a_{2^n-1}$ and the $a_{2^{n_i}-1}$ are independent of one another, this case occurs precisely when the coefficients of $a_{2^n-1}$ and the $a_{2^{n_i}-1}$ are zero, and when $y_{2^{n+1}-2}+R_{2^{n+1}-3}$ is zero.  The value of $R_{2^{n+1}-2-(2^n-1)-1}$ depends only on $(x_0,y_0,...,x_{2^n-2},y_{2^n-2})$; in particular, it is independent of $x_{2^{n+1}-2^{n}-1}$.  Hence the coefficient of $a_{2^n-1}$ is zero with probability $\frac{1}{q}$, as this occurs precisely when $x_{2^{n+1}-2^{n}-1}=-R_{2^{n+1}-2-(2^{n}-1)-1}$.
\\
Similarly, for each $i\leq l$, the value of $R_{2^{n+1}-2-(2^{n_i}-1)-1}$ depends only on $x_0,y_0,...,x_{2^{n+1}-2^{n_i}-2},y_{2^{n+1}-2^{n_i}-2}$, and so the coefficient of $a_{2^{n_l}-1}$ is zero precisely when $x_{2^{n+1}-2^{n_l}-1}=-R_{2^{n+1}-2-(2^{n_l}-1)-1}$.  This also occurs with probability $\frac{1}{q}$ for the same reason as above.
\\
Finally, $R_{2^{n+1}-3}$ depends only on $x_0,y_0,...,x_{2^{n+1}-3},y_{2^{n+1}-3}$, and so $y_{2^{n+1}-2}+R_{2^{n+1}-3}$ equals zero with probability $\frac{1}{q}$.
\\
Since each event depends on different $x_i$ or $y_i$, they are collectively independent.  Therefore, Expression \ref{newLHS} is identically zero with probability $\frac{1}{q^{l+2}}$.

\item \textbf{\eqref{newLHS} is never zero.} This case occurs precisely when the coefficients of $a_{2^n-1}$ and the $a_{2^{n_i}-1}$ are zero, and when $y_{2^{n+1}-2}+R_{2^{n+1}-3}$ is nonzero.  By the same arguments as in case (1), the coefficient $a_{2^n-1}$ and each $a_{2^{n_i}-1}$ is zero with probability $\frac{1}{q}$, and $y_{2^{n+1}-2}+R_{2^{n+1}-3}$ is nonzero with probability $\frac{q-1}{q}$.  Again, all events are independent, hence Expression \ref{newLHS} is never zero with probability $\frac{q-1}{q^{l+2}}$.

\item\textbf{In \eqref{newLHS}, the coefficient of $a_{2^n-1}$ or one of the $a_{2^{n_i}-1}$ is not zero.} Since this is the only remaining case, its probability is $1-\frac{q-1}{q^{l+2}}-\frac{1}{q^{l+2}}=1-\frac{1}{q^{l+1}}$.

\end{enumerate}
Now, if (1) holds, then $s_{2^{n+1}-2}(x,y)=q^{l+1}$, since $a_{2^{n}-1}$ will be linearly independent of the other $a_{2^i - 1}$.  If (2) holds, then $s_{2^{n+1}-2}(x,y)=0$.  Finally, if (3) holds, then $a_{2^{n+1}-2}$ is determined as a linear function of the other $a_{2^i - 1}$, and so $s_{2^{n+1}-2}(x,y)=q^{l}$.
\end{proof}

Now we can prove Theorem \ref{thm_construction_of_specific_Kakeya}:

\begin{proof}[Proof of Theorem \ref{thm_construction_of_specific_Kakeya}] Observe that the probabilities computed in Lemma \ref{probabilities} are compatible with the Haar measure on $\F_q[[t]]^2$, and also that those probabilities depend only on $l$ and not on $n$. Therefore we can create a Markov chain on the points $0,1,q,q^2,...$, and the related random variables described by $$P(X_n=l)=\mu(\{(x,y)\in \F_q[[t]]^2 \, : \, s_{2^n-2}(x,y)=l\}),$$
with transition functions as follows: at time $n$, from the point $q^l$ we go to one of $0$, $q^l$ and $q^{l+1}$, with probabilities $\frac{q-1}{q^{l+2}}$, $1-\frac{1}{q^{l+1}}$, $\frac{1}{q^{l+2}}$ respectively.

Observe that the expected value does not change:
$$E(X_n|X_{n-1}=l)=0\cdot\frac{q-1}{q^{l+2}}+q^l\cdot\left(1-\frac{1}{q^{l+1}}\right)+q^{l+1}\cdot\frac{1}{q^{l+2}}=q^l-\frac{1}{q}+\frac{1}{q}=q^l.$$

Since from every state a positive proportion (independent of $n$) is sent to $0$, we see by a trivial induction on $l$ that as $n \to \infty$, the proportion of states at $q^l$ tends to 0, for each $l\geq 0$.  Combined with the observation that the expected value does not change, we see that the limiting distribution is concentrated at $0$.  Hence $\mu(\{(x,y)|s_{2^n-2}(x,y)=0\})\rightarrow 1$, and the proof is complete.
\end{proof}
Combining the results of Lemma \ref{general_Kakeya_from_specific} and Theorem \ref{thm_construction_of_specific_Kakeya}, we obtain Theorem \ref{main}:
\begin{thmcustom}{\ref{main}}
For all $n>1$, there exists a subset $E\subset \F_q[[t]]^n$ of measure 0 which contains a line in every direction.
\end{thmcustom}

\begin{rem}
By using an inductive argument, one may prove that, for $q>2$, at time $t$, the total measure $\mu_q(t)$ of the Kakeya set in $\mathbb{F}_q[[t]]^2$ satisfies 
$$\mu_q(t) \ll \dfrac{\ln(t)}{t},$$
where the (explicitly computable) constant depends only on $q$.  In particular, we have
$$|H_k||\F_q|^{-2k} \ll \dfrac{\ln(\ln(k))}{\ln(k)},$$
where $H_k$ denotes the size of the projection of the Kakeya set to $\F_q[[t]]/t^k$.  Taking the limit as $k\to\infty$ yields 
$$\lim_{k\to\infty} |H_k||\F_q|^{-2k} = 0$$
which answers Question 1 as posed in \cite{Ell09}.
\end{rem}

\section{Minkowski dimension for $R=\F_q[[t]]$ and $R=\Z_p$}

The analogue of the Minkowski dimension of a subset $E$ of $\F_q[[t]]^n$ is 
$$\lim_{k\rightarrow \infty} \frac{\log |E_k|}{\log q^k}$$
where $E_k$ is the image of $E$ under the projection $\F_q[[t]]\rightarrow \F_q[[t]]/t^k$. In this section we prove that if $E$ is a Kakeya set in $\F_q[[t]]^2$ then the limit exists and equals 2. In order to obtain this result we give a sharper lower bound than Proposition 4.21 in \cite{Ell09} for the measure of any Kakeya set in $R^2$ where $R=\F_q[t]/t^k$ or $\Z/p^k\Z$:

\begin{prop}\label{finite}
Let $E$ be a Kakeya set in $R^2$ where $R=\F_q[t]/t^k$ or $\Z/p^k\Z$. Then
$$|E|\geq \frac{|R|^2}{2k}.$$
\end{prop}

Let $E$ be Kakeya; then it contains a line of the form $ax+y=b$ for all $a\in R$. We enumerate the lines by their direction vectors as follows:
\begin{itemize}
 \item If $R=\Z/p^k\Z$ then we take $L_i$ to be (any) of the lines in $E$ of the form $\alpha_i x+y=b_i$, where $\alpha_i = i \pmod{p^k}$.
 \item If $R=\F_q[t]/t^k$ then we take $L_i$ to be (any) of the lines in $E$ of the form $\alpha_i x+y=b_i \hspace{0.2 cm}(\mbox{mod }t^k)$, where, for $i=\sum_{j=0}^{k-1}a_jq^j$ in base $q$, we have $\alpha_i = \left(\sum_{j=0}^{k-1}a_jt^j\right)$.
\end{itemize}

\begin{nota}
For nonzero $r \in R$, we denote by $v(r)$ the largest integer satisfying $r\in \mathfrak{m}^{v(r)}$, where $\mathfrak{m}$ is the nilradical of $R$.
\end{nota}

For any two lines $L_i:\alpha_i x+y=b_i$ and $L_j:\alpha_j x+y=b_j$ with $i\ne j$, we see that if $(x_0,y_0) \in L_i \cap L_j$ then $(\alpha_i-\alpha_j)x_0=b_i-b_j$. Hence if $v(\alpha_i-\alpha_j)=l$, then the number of possible values for $x_0$ cannot exceed $|\mathfrak{m}|^l$. In particular, since the value of $x_0$ determines the value of $y_0$, the size of the intersection $|L_i \cap L_j|$ is at most $|\mathfrak{m}|^l = |\mathfrak{m}|^{v(\alpha_i-\alpha_j)}$.

\begin{lem}\label{function_bounds}
For the function $f(u):=\sum_{i=1}^{u} \mathfrak{m}^{v(\alpha_i)}$, we have $f(u) \leq u\cdot \left\lceil\log_{|\mathfrak{m}|} (u)\right\rceil$.  In particular, for $u\leq|\mathfrak{m}|^k / k = |R| / k$, we have $f(u) \leq |R|$.
\end{lem}
\begin{proof}
There are $\left\lfloor\frac{u}{|\mathfrak{m}|^w}\right\rfloor$ terms for which $v(\alpha_i)$ is equal to  $|\mathfrak{m}|^w$.  Thus we may write 
\begin{align*}
f(u) &= \sum_{w=0}^{\infty} (|\mathfrak{m}|^w - |\mathfrak{m}|^{w-1}) \cdot \left\lfloor\frac{u}{|\mathfrak{m}|^w}\right\rfloor \\
  &\leq \sum_{w=0}^{\left\lceil\log_{|\mathfrak{m}|} (u)\right\rceil} (|\mathfrak{m}|^w) \cdot \frac{u}{|\mathfrak{m}|^w} \\
  &= u\cdot \left\lceil\log_{|\mathfrak{m}|} (u)\right\rceil,
\end{align*}
which gives the first statement.  The second statement follows immediately.
\end{proof}

With this bound, we may now prove Proposition \ref{finite}.
\begin{proof}[Proof of Proposition \ref{finite}] \textrm{\,} \\
Inclusion-exclusion gives a lower bound on the size of $E$:
\begin{equation}\label{basic} 
|E|\geq |\cup_{j=1}^{l+1} L_j| \geq \sum_{j=1}^{l+1} \left( |L_j|-\sum_{i=1}^{j-1} |L_i\cap L_j| \right),
\end{equation}
where $l=\left\lfloor\frac{|R|}{k}\right\rfloor$ is as given in Lemma \ref{function_bounds}.
We can explicitly write the terms in \eqref{basic} as
\begin{equation*}\label{est}
|L_j|-\sum_{i=1}^{j-1} |L_i\cap L_j|=|R|-\sum_{i=1}^{j-1} \mathfrak{m}^{v(\alpha_j - \alpha_i)} = |R|-f(j-1).
\end{equation*}
Applying the upper bound on $f(u)$ from Lemma \ref{function_bounds} yields
\begin{align*}
|E| &\geq \sum_{j=1}^{l+1} \left(|R|-f(j-1)\right) \\
    &= (l+1)|R| - \sum_{j=2}^{l+1} (j-1) \left\lceil\log_{|\mathfrak{m}|} (j-1)\right\rceil \\
    &\geq (l+1)|R| -  (\left\lceil\log_{|\mathfrak{m}|} (l)\right\rceil) \frac{l(l+1)}{2} \\
    &\geq |R| \frac{l+1}{2} \\
    &> \frac{|R|^2}{2k},
\end{align*}
and this is precisely the desired bound.
\end{proof}

Now we can complete the proof of Theorem \ref{Mink}:
\begin{proof} [Proof of Theorem \ref{Mink}]
From Proposition \ref{finite}, if $E$ is a Kakeya set in $R^2$ with either $R=\Z/p^k\Z$ or $R=\F_q[t]/t^k$, then the Minkowski dimension of $E$ is at least $2-\frac{\log 2k}{k \log|\mathfrak{m}|}$. In particular, if we fix $p$ or $q$ and take $k \to \infty$ then the lower bound goes to 2.  The desired result follows, since the Minkowski dimension of any set in $R^2$ is at most 2.
\end{proof}

\section{Acknowledgements}
The authors would like to thank Jordan Ellenberg, Richard Oberlin and Terence Tao for their original paper which inspired our work, and Jordan Ellenberg in particular for his guidance and persistent encouragement in preparing this manuscript.


\begin{thebibliography}{99}

\bibitem[Bes19]{Bes19} Abram Besicovitch \textit{Sur deux questions d'integrabilite des fonctions}, J. Soc. Phys. Math. Volume 2, 1919, pp 105-123.

\bibitem[Dav71]{Dav71} Roy O. Davies, \textit{Some remarks on the Kakeya problem}, Proc. Cambridge Phil. Soc. Volume 69, 1971, pp 417-421.

\bibitem[Ell09]{Ell09} J. S. Ellenberg, R. Oberlin and T. Tao, \textit{The Kakeya set and maximal conjectures for algebraic varieties over finite fields}, Mathematika, Volume 56, Issue 01, January 2010, pp 1-25.

\bibitem[Wol99]{Wol99} Thomas Wolff. \textit{An improved bound for Kakeya type maximal functions}, Rev. Mat. Iberoamericana Volume, 11, 1999, pp 651-674.

\bibitem[Dvir08]{Dvir08} Zeev Dvir.  \textit{On the size of Kakeya sets in finite fields}, J. Amer. Math. Soc, 2009.

\end{thebibliography}
\end{document}